\documentclass[11pt]{article}

\usepackage[T1]{fontenc}
\usepackage[utf8]{inputenc}
\usepackage{microtype}

\usepackage{amsmath,amsthm,amssymb,mathtools}
\usepackage{bm}
\usepackage{enumitem}
\usepackage[dvipsnames]{xcolor}
\usepackage{hyperref}
\hypersetup{
  colorlinks=true,
  linkcolor=blue!50!black,
  citecolor=blue!50!black,
  urlcolor=blue!50!black,
  pdftitle={A Refined Parameter Condition in the Lyapunov Analysis of IGAHD},
  pdfauthor={Samir Adly},
  pdfsubject={IGAHD, Hessian-driven damping, and discrete Lyapunov analysis},
  pdfkeywords={IGAHD, Hessian-driven damping, accelerated optimization,
    Lyapunov function, smooth convex interpolation}
}
\usepackage[a4paper,margin=0.8in]{geometry}
\usepackage{booktabs}
\usepackage{tabularx}
\usepackage{xspace}

\theoremstyle{plain}
\newtheorem{theorem}{Theorem}[section]
\newtheorem{proposition}[theorem]{Proposition}

\newtheorem{corollary}[theorem]{Corollary}
\theoremstyle{definition}
\newtheorem{remark}[theorem]{Remark}

\usepackage{titlesec}
\titlespacing*{\section}
  {0pt}
  {1.8ex plus 0.4ex minus 0.2ex}
  {0.8ex plus 0.2ex}

\titlespacing*{\subsection}
  {0pt}
  {1.4ex plus 0.3ex minus 0.2ex}
  {0.6ex plus 0.2ex}
\newcommand{\R}{\mathbb{R}}
\newcommand{\Hc}{\mathcal H}
\newcommand{\norm}[1]{\left\lVert #1\right\rVert}
\newcommand{\inner}[2]{\left\langle #1,#2\right\rangle}
\newcommand{\IGAHD}{\textup{IGAHD}\xspace}
\DeclareMathOperator*{\argminop}{arg\,min}
\newcolumntype{Y}{>{\raggedright\arraybackslash}X}

\AtBeginDocument{
  \setlength{\abovedisplayskip}{6pt plus 2pt minus 2pt}
  \setlength{\belowdisplayskip}{7pt plus 2pt minus 2pt}
  \setlength{\abovedisplayshortskip}{4pt plus 1pt}
  \setlength{\belowdisplayshortskip}{6pt plus 1pt}
}

\makeatletter
\renewenvironment{abstract}
  {\par\smallskip\small\noindent\textbf{Abstract. }\ignorespaces}
  {\par\medskip}
\makeatother

\title{A Refined Parameter Condition in the Lyapunov Analysis of IGAHD}
\author{
Samir Adly\thanks{Laboratoire XLIM, Universit\'e de Limoges,
123 avenue Albert Thomas, 87060 Limoges CEDEX, France.
Email: \href{mailto:samir.adly@unilim.fr}{samir.adly@unilim.fr},
\url{https://www.unilim.fr/pages_perso/samir.adly/}.}
}
\date{}

\begin{document}
\maketitle
\vskip -8mm
\begin{abstract} In a recent paper, Attouch, Chbani, Fadili and Riahi introduced the inertial gradient algorithm with Hessian-driven daming, called (IGAHD). Under the condition $0\leq\beta<2\sqrt{s}$, where $\beta$ is the Hessian driven damping parameter and $s>0$ is the gradient step size, their Lyapunov analysis shows an accelerated estimate of the objective function as well as a weighted summability of the gradient for the case $\beta>0$. For any fixed value of $\beta>0$, this condition excluded sufficiently small step sizes $s>0$. The present note refines one of the estimates in this Lyapunov analysis by retaining two coefficients that were previously replaced by lower bounds. This leads to the following less restrictive and sufficient condition:
$$0<\beta L\sqrt{s}<1+\sqrt{1+sL(1-sL)},$$
where $L>0$ is the constant Lipschitz of the gradient objective function and $0<s\leq1/L$. In particular, the simplest condition $0<\beta <\frac{2}{L\sqrt{s}}$ is sufficient. For fixed values of $\beta>0$ and $L>0$, this new condition is satisfied for sufficiently small $s>0$. Consequently, under this new refined sufficient condition, the conclusions in the original paper stay valid. 
For convex quadratic functions in finite dimensions, a separate spectral analysis leads to a larger region parameters. It gives a necessary and sufficient condition for Schur stability of the limit modal matrices, as well as the geometrical decay of the objective residuals and the gradients. This modal analysis raises the question of whether a different Lyapunov function would allow us to recover part (or all) of this extended spectral domain for non-quadratic objective functions.
\end{abstract}



\tableofcontents
%=====================================================
\section{Introduction}\label{sec:introduction}
%=====================================================
For a smooth objective function $f$ in a real Hilbert space $\Hc$ and a trajectory $t\mapsto x(t)$, the Hessian-driven damping with coefficient $\beta>0$ involves the term $\beta\nabla^2 f(x(t))\dot{x}(t)=\beta \frac{d}{dt}\nabla f(x(t))$ in the associated continuous time dynamical system. This can be interpreted as a  geometrical damping, and is known to reduce the oscillations induced by the inertia while preserving a first order discretization optimization algorithm. In fact, its discretization is based on the difference of two successive gradients instead of the explicit evaluation of the Hessian matrix at the current point. This observation led Attouch, Chbani, Fadili  and Riahi~\cite{AFCR2022} to introduce and study the Inertial Gradient Algorithm with Hessian Damping (IGAHD).

The construction of the algorithm \IGAHD, the auxiliary sequence, the Lyapunov function, and the cancellation identity used below are due to \cite{AFCR2022}. 
The contribution of this note regarding the Lyapunov analysis is limited to refining an estimate in the original proof. A separate spectral analysis is proposed for finite-dimensional convex quadratic functions to compare the two conditions.

For a convex function with an $L$-Lipschitz gradient, the analysis in \cite[Theorem~6]{AFCR2022} establishes the estimate $f(x_k)-\min_{\mathcal{H}}f=O(k^{-2})$ under the conditions $sL\leq 1$, $\alpha\geq 3$, and $0\leq\beta<2\sqrt{s}$. This sufficient condition on $\beta$ is incompatible with small step sizes $s>0$ when $\beta>0$ is fixed. Indeed, if $s=h^2$, this condition becomes $\beta<2h$ and consequently fails for any $h\leq\beta/2$. The original paper \cite{AFCR2022} explicitly left open the status of this restriction: ``It is an open question whether this constraint is a technical artifact or is fundamental to acceleration. We leave it to a future work.'' \cite[Remark~3]{AFCR2022}.

The original restriction is also used in subsequent works. The condition $\beta<2\sqrt{s}$ is for example recalled in \cite[Section~4]{ADHLR2023}. The authors then numerically explore the choices $\alpha=3$ and $\beta\in\{100,300,500\}/\sqrt{L}$, explicitly noting that this regime is not covered by available convergence results.

The section of this note devoted to the Lyapunov approach focuses on the role of this condition in the proof. A less restrictive sufficient condition arises from the same Lyapunov construction when two intermediate coefficients are kept at their exact values. Proposition~\ref{prop:gradient-step} retains two coefficients that had been replaced by $s/2$ in the original estimate for the smooth convex case \cite[Appendix~A.1, Lemma~1]{AFCR2022}. 
The resulting dissipation is governed by a family of symmetric $2\times2$ matrices. Under the condition
$$
(\beta L\sqrt{s}-1)^2<1+sL(1-sL),
$$
the normalized matrices converge to a positive definite limit. Consequently, the finite-index matrices are positive definite for all sufficiently large indices.

For $\beta>0$, this is equivalent to
\begin{equation}\label{new-condition-beta}
0<\beta L\sqrt{s}<1+\sqrt{1+sL(1-sL)}.
\end{equation}
Note that a simpler sufficient condition than \eqref{new-condition-beta} is given by $0<\beta L\sqrt{s}<2$.
Thanks to this single modification, the Lyapunov argument from \cite{AFCR2022} yields the same asymptotic conclusions over a wider parameter range. The estimate on the objective function and the weighted summability of gradients constitute the conclusions of \cite[Theorem~6]{AFCR2022}; the bounds on the best iterate already appear in \cite[Remark~4]{AFCR2022}. We restate them below solely to clarify the consequences of the refined condition on the parameters.

For $\alpha > 3$, the refined analysis also gives the weighted summability of the objective function required in \cite{ACFR2023}. As noted in Remark \ref{rem:acfr2023-extension}, the arguments from Theorems 4.2 to 4.4 of \cite{ACFR2023} then allow for obtaining the same rate estimates, the weak convergence of $(x_k)$ to a point in $X^\star$, and the other asymptotic conclusions under the new condition \eqref{new-condition-beta}. The convergence arguments themselves remain those of \cite{ACFR2023}.

The refined Lyapunov condition is sufficient. No claim regarding its necessity is made for the full class of smooth convex functions. In the case of convex quadratic objective functions in finite dimensions, the spectral analysis presented in Subsection \ref{subsec:quadratic} allows for extending the parameter range, including the case $\alpha=3$, and establishes geometrical decay of the objective function residual and the gradients. The statement concerning necessity applies only to the Schur stability of the limiting modal matrices.

The IGAHD algorithm has been also compared to Nesterov's accelerated gradient method and the Ravine-type accelerated gradient method. The analysis presented in \cite{AAF2024} draws on developments regarding high and super resolution ordinary differential equations (ODEs) and includes some numerical experiments. Accelerated restart strategies for the underlying dynamics with Hessian-driven damping, as well as algorithmic illustrations for IGAHD, are investigated in \cite{MaulenPeypouquet2023}. These worsk address issues distinct from the parameter analyses considered here.

The note is organized as follows. Section~\ref{sec:setting} recalls the algorithm and states the refined Lyapunov condition. Section~\ref{sec:local-estimate} establishes the one-step estimate with its exact coefficients. Section~\ref{sec:lyapunov} calculates the Lyapunov increment, and Section~\ref{sec:positivity} examines the resulting quadratic form.
Section~\ref{sec:comparison} first compares the conditions regarding the Lyapunov parameters. It then presents a spectral analysis specific to convexe quadratic objectives.
Section~\ref{sec:scope} examines limiting cases and specifies the domain of validity for both arguments.
%=====================================================
\section{Setting and refined parameter condition}\label{sec:setting}
%=====================================================
Let $\Hc$ be a real Hilbert space, and let $f:\Hc\to\R$ be convex and
Fr\'echet differentiable. We assume that $\nabla f$ is $L$-Lipschitz
continuous for some $L>0$ and that
$$
X^\star:=\argminop_{\Hc}f\neq\varnothing,
\qquad
f^\star:=\min_{\Hc}f.
$$
Fix $\alpha\geq3$, $0<s\leq1/L$, and $\beta\geq0$. Starting from
$x_0,x_1\in\Hc$, \IGAHD is defined by
\begin{equation}
\label{eq:igahd}
\left\{
\begin{aligned}
a_k&=1-\frac{\alpha}{k},\\
y_k&=x_k+a_k(x_k-x_{k-1})
-\beta\sqrt{s}\bigl(\nabla f(x_k)-\nabla f(x_{k-1})\bigr)
-\frac{\beta\sqrt{s}}{k}\nabla f(x_{k-1}),\\
x_{k+1}&=y_k-s\nabla f(y_k),
\end{aligned}
\right.
\qquad k\geq1.
\end{equation}

We use the following dimensionless quantities:
\begin{equation}
\label{eq:dimensionless}
q:=sL\in(0,1],
\qquad
\delta:=\beta L\sqrt{s},
\qquad
C_q:=1+q(1-q).
\end{equation}
Define
\begin{equation}
\label{eq:tau}
\tau_k:=\frac{k-1}{\alpha-1},
\end{equation}
and, for a fixed $x^\star\in X^\star$,
\begin{equation}
\label{eq:vk}
v_k:=x_{k-1}-x^\star
+\tau_k\bigl(x_k-x_{k-1}+\beta\sqrt{s}\nabla f(x_{k-1})\bigr).
\end{equation}
The Lyapunov function is defined by
\begin{equation}
\label{eq:energy}
\mathcal E_k
:=\tau_k^2\bigl(f(x_k)-f^\star\bigr)+\frac{1}{2s}\norm{v_k}^2.
\end{equation}
These weights, this auxiliary vector, and this energy are the ones used in the proof of \cite[Theorem~6]{AFCR2022}.

The refined condition gives the same convergence properties on a larger parameter region. Apart from this larger sufficient region and the explicit
statement of eventual monotonicity, the conclusions below are those obtained in \cite[Theorem~6 and Remark~4]{AFCR2022}.

The following theorem states the precise condition used in the rest of the note.

\begin{theorem}
\label{thm:main}
Let $(x_k)$ be generated by \eqref{eq:igahd}, with $\alpha\geq3$, $0<s\leq1/L$, and $\beta>0$. Suppose that
\begin{equation}
\label{eq:main-condition}
0<\beta L\sqrt{s}<1+\sqrt{1+sL(1-sL)}.
\end{equation}
Then the energy $(\mathcal E_k)$ is nonincreasing from some index onward and
\begin{equation}
\label{eq:objective-rate}
f(x_k)-f^\star=O(k^{-2}).
\end{equation}
Moreover,
\begin{equation}
\label{eq:gradient-summability}
\sum_{k=1}^{+\infty}k^2\norm{\nabla f(x_k)}^2<+\infty,
\qquad
\sum_{k=1}^{+\infty}k^2\norm{\nabla f(y_k)}^2<+\infty.
\end{equation}
Consequently,
\begin{equation}
\label{eq:gradient-little-o}
\norm{\nabla f(x_k)}=o(k^{-1}),
\qquad
\norm{\nabla f(y_k)}=o(k^{-1}),
\end{equation}
and
\begin{equation}
\label{eq:best-iterate}
\min_{1\leq i\leq k}\norm{\nabla f(x_i)}^2=O(k^{-3}),
\qquad
\min_{1\leq i\leq k}\norm{\nabla f(y_i)}^2=O(k^{-3}).
\end{equation}
\end{theorem}

Since $q(1-q)\geq0$ for $q\in(0,1]$, condition \eqref{eq:main-condition} is satisfied under the simpler requirement
\begin{equation}
\label{eq:simple-condition-intro}
0<\beta L\sqrt{s}<2.
\end{equation}
%=====================================================
\section{The local refinement: retaining the exact coefficients} \label{sec:local-estimate}
%=====================================================
We begin with the smooth convex interpolation inequality. In finite dimensions, it follows from
\cite[Theorem~2.1.5, inequality~(2.1.7)]{Nesterov} by interchanging $x$ and $y$. We give a short proof in the Hilbert-space setting because
the exact coefficient in this inequality will be used below. The inequality that will be used is the following: for every $x,y\in\Hc$,
\begin{equation}
\label{eq:interpolation}
f(y)\leq f(x)+\inner{\nabla f(y)}{y-x}
-\frac{1}{2L}\norm{\nabla f(x)-\nabla f(y)}^2.
\end{equation}
Indeed, fix $y\in\Hc$ and define
$$
h(z):=f(z)-\inner{\nabla f(y)}{z}.
$$
The function $h$ is convex and Fr\'echet differentiable, its gradient is
$L$-Lipschitz, and $
\nabla h(y)=0.$
Thus, $y$ minimizes $h$. {Applying the standard descent lemma, which is valid in Hilbert spaces by integration along line segments, to $h$ at $x$
with step size $1/L$ gives}
$$
h\left(x-\frac{1}{L}\nabla h(x)\right)
\leq h(x)-\frac{1}{2L}\norm{\nabla h(x)}^2.
$$
Since $y$ minimizes $h$, it follows that
$$
h(y)\leq
h\left(x-\frac{1}{L}\nabla h(x)\right)
\leq h(x)-\frac{1}{2L}\norm{\nabla h(x)}^2.
$$
Finally,
$$
\nabla h(x)=\nabla f(x)-\nabla f(y),
$$
and expanding the definition of $h$ yields \eqref{eq:interpolation}.

We now combine \eqref{eq:interpolation} with the descent step while retaining the exact coefficients.

\begin{proposition}
\label{prop:gradient-step}
Let $0<s\leq1/L$ and set $q=sL$. For $x,y\in\Hc$, define
$$
X:=\nabla f(x),
\qquad
Y:=\nabla f(y).
$$
Then
\begin{equation}
\label{eq:gradient-step}
f(y-sY)
\leq f(x)+\inner{Y}{y-x}
-a\norm{Y}^2-b\norm{X-Y}^2,
\end{equation}
where
\begin{equation}
\label{eq:ab}
a:=s\left(1-\frac q2\right),
\qquad
b:=\frac{1}{2L}.
\end{equation}
\end{proposition}

\begin{proof}
The descent lemma at $y$ gives
$$
f(y-sY)
\leq f(y)-s\norm{Y}^2+\frac{Ls^2}{2}\norm{Y}^2
=f(y)-s\left(1-\frac q2\right)\norm{Y}^2.
$$
Applying \eqref{eq:interpolation} and using
$X=\nabla f(x)$ and $Y=\nabla f(y)$ yields
$$
f(y-sY)
\leq f(x)+\inner{Y}{y-x}
-s\left(1-\frac q2\right)\norm{Y}^2
-\frac{1}{2L}\norm{X-Y}^2.
$$
This is \eqref{eq:gradient-step}.
\end{proof}
%%%%%%%%%%%%%%%%%%%%%%%%%%%%%%%
\begin{remark}\normalfont
The relation with \cite[Appendix~A.1, Lemma~1]{AFCR2022} is as follows. Inequality \eqref{eq:interpolation} is not
the statement of that lemma, but it appears as an intermediate estimate in its proof. More precisely, before weakening the coefficients, the
proof in \cite{AFCR2022} obtains
$$
f(y-sY)
\leq f(y)-s\left(1-\frac q2\right)\norm{Y}^2
$$
from the standard descent lemma, and
$$
f(y)\leq f(x)+\inner{Y}{y-x}
-\frac{1}{2L}\norm{X-Y}^2
$$
from smooth convex interpolation. The assumptions $q=sL\leq1$ and
$s\leq1/L$ imply
$$
s\left(1-\frac q2\right)\geq\frac s2,
\qquad
\frac{1}{2L}\geq\frac s2.
$$
Replacing both exact coefficients by $s/2$ gives the estimate stated in their Lemma 1:
$$
f(y-sY)
\leq f(x)+\inner{Y}{y-x}
-\frac s2\norm{Y}^2
-\frac s2\norm{X-Y}^2.
$$
This weakened estimate is then used in the proof of
\cite[Theorem~6]{AFCR2022}.

{The estimate in Proposition~\ref{prop:gradient-step} has the same left-hand side and the same affine term as the estimate in
\cite[Appendix~A.1, Lemma~1]{AFCR2022}. The difference lies in the two dissipative coefficients: Proposition~\ref{prop:gradient-step} retains
$a$ and $b$, whereas the estimate in \cite{AFCR2022} uses the common lower bound $s/2$.}

The corresponding margins are
\begin{equation}
\label{eq:margins}
a-\frac s2=\frac{s(1-q)}{2},
\qquad
b-\frac s2=\frac{1-q}{2L}.
\end{equation}
They are nonnegative for $0<s\leq1/L$ and vanish simultaneously when $s=1/L$.

{Consequently, Proposition~\ref{prop:gradient-step} implies the estimate in \cite[Appendix~A.1, Lemma~1]{AFCR2022}. The two estimates
coincide when $s=1/L$. When $0<s<1/L$, both coefficient inequalities are strict, so Proposition~\ref{prop:gradient-step} is strictly sharper at
the level of the dissipative coefficients.}

Thus, the interpolation inequality itself is not new. The refinement consists in retaining its exact coefficient together with the exact
coefficient produced by the descent step.

For later use, let $x^\star\in X^\star$. Since $\nabla f(x^\star)=0$, Proposition~\ref{prop:gradient-step} applied with $x=x^\star$ gives
\begin{equation}
\label{eq:gradient-step-minimizer}
f\bigl(y-s\nabla f(y)\bigr)-f^\star
\leq
\inner{\nabla f(y)}{y-x^\star}
-(a+b)\norm{\nabla f(y)}^2,
\end{equation}
with $a$ and $b$ defined in \eqref{eq:ab}.
\end{remark}
%=====================================================
\section{Revisiting the Lyapunov increment of \texorpdfstring{\cite{AFCR2022}}{[1]}}
\label{sec:lyapunov}
%=====================================================
The parameters $\tau_k$ in \eqref{eq:tau} satisfy
\begin{equation}
\label{eq:tau-identities}
\tau_{k+1}=\frac{k}{\alpha-1},
\qquad
\tau_k=1+a_k\tau_{k+1},
\qquad
\tau_k=\left(1-\frac1k\right)\tau_{k+1}.
\end{equation}
Since $\alpha\geq3$,
\begin{equation}
\label{eq:tau-inequality}
\tau_{k+1}^2-\tau_{k+1}-\tau_k^2
=\frac{k(3-\alpha)-1}{(\alpha-1)^2}\leq0.
\end{equation}

The cancellation identity associated with \eqref{eq:vk} is established in the proof of \cite[Theorem~6]{AFCR2022}. In the present
notation, it reads
\begin{equation}
\label{eq:v-cancellation}
v_{k+1}-v_k=-s\tau_{k+1}\nabla f(y_k).
\end{equation}
Indeed, the calculation in \cite{AFCR2022} leaves the coefficient
$$
\beta\sqrt{s}
\left[\tau_{k+1}\left(1-\frac1k\right)-\tau_k\right]
$$
in front of $\nabla f(x_{k-1})$, which vanishes by the last identity in
\eqref{eq:tau-identities}.

We now repeat the Lyapunov increment calculation of \cite[Theorem~6]{AFCR2022} with
Proposition~\ref{prop:gradient-step}. We include the details only to isolate the additional terms supplied by the margins in \eqref{eq:margins}; all other parts of the calculation follow the original proof.

This calculation leads to an energy bound that isolates the refined dissipation term.

\begin{proposition}
\label{prop:energy-increment}
Let $k$ be such that $\tau_{k+1}\geq1$, and set
$$
\tau:=\tau_{k+1},
\qquad
X:=\nabla f(x_k),
\qquad
Y:=\nabla f(y_k).
$$
Then
\begin{equation}
\label{eq:energy-increment}
\mathcal E_{k+1}-\mathcal E_k\leq-\tau\widetilde B_k,
\end{equation}
where
\begin{equation}
\label{eq:B-unexpanded}
\widetilde B_k
=\tau\beta\sqrt{s}\inner{Y}{X}
+\left(\frac{\tau s(1-q)}2+\frac{1}{2L}\right)\norm{Y}^2
+\frac{\tau-1}{2L}\norm{X-Y}^2.
\end{equation}
\end{proposition}

\begin{proof}
Apply Proposition~\ref{prop:gradient-step} with $x=x_k$, $y=y_k$, and
$y_k-sY=y_k-s\nabla f(y_k)=x_{k+1}$:
\begin{equation}
\label{eq:first-function}
f(x_{k+1})-f^\star
\leq f(x_k)-f^\star+\inner{Y}{y_k-x_k}
-a\norm{Y}^2-b\norm{X-Y}^2.
\end{equation}
Equation~\eqref{eq:gradient-step-minimizer} gives
\begin{equation}
\label{eq:second-function}
f(x_{k+1})-f^\star
\leq\inner{Y}{y_k-x^\star}-(a+b)\norm{Y}^2.
\end{equation}
{Multiplying \eqref{eq:first-function} by $\tau-1\geq0$ and adding
\eqref{eq:second-function} gives}
\begin{align}
\tau\bigl(f(x_{k+1})-f^\star\bigr)
\leq{}&(\tau-1)\bigl(f(x_k)-f^\star\bigr)
+\inner{Y}{(\tau-1)(y_k-x_k)+y_k-x^\star}\notag\\
&-(\tau a+b)\norm{Y}^2-(\tau-1)b\norm{X-Y}^2.
\label{eq:weighted-function}
\end{align}
{Multiplying by $\tau$, using \eqref{eq:tau-inequality}, and recalling
that $f(x_k)-f^\star\geq0$ gives}
\begin{align}
\tau^2\bigl(f(x_{k+1})-f^\star\bigr)
\leq{}&\tau_k^2\bigl(f(x_k)-f^\star\bigr)
+\tau\inner{Y}{(\tau-1)(y_k-x_k)+y_k-x^\star}\notag\\
&-\tau(\tau a+b)\norm{Y}^2
-\tau(\tau-1)b\norm{X-Y}^2.
\label{eq:weighted-function-final}
\end{align}
The identity
$$
\frac12\norm{u}^2-\frac12\norm{v}^2
=\inner{u-v}{u}-\frac12\norm{u-v}^2
$$
and \eqref{eq:v-cancellation} give
\begin{equation}
\label{eq:kinetic-increment}
\frac{1}{2s}\bigl(\norm{v_{k+1}}^2-\norm{v_k}^2\bigr)
=-\tau\inner{Y}{v_{k+1}}-\frac{s\tau^2}{2}\norm{Y}^2.
\end{equation}
After adding \eqref{eq:weighted-function-final} and
\eqref{eq:kinetic-increment}, the vector paired with $Y$ is
$$
(\tau-1)(y_k-x_k)+y_k-x^\star-v_{k+1}=\tau\bigl(y_k-x_{k+1}-\beta\sqrt{s}X\bigr)
=\tau\bigl(sY-\beta\sqrt{s}X\bigr).
$$
The combined inner products are therefore
\begin{equation}
\label{eq:inner-combination}
\tau^2s\norm{Y}^2-\tau^2\beta\sqrt{s}\inner{Y}{X}.
\end{equation}
Collecting \eqref{eq:weighted-function-final},
\eqref{eq:kinetic-increment}, and \eqref{eq:inner-combination} gives
$$
\mathcal E_{k+1}-\mathcal E_k
\leq{}-\tau^2\beta\sqrt{s}\inner{Y}{X}
+\left(\frac{s\tau^2}{2}-\tau^2a-\tau b\right)\norm{Y}^2-\tau(\tau-1)b\norm{X-Y}^2.
$$
Substitution of \eqref{eq:ab} proves
\eqref{eq:energy-increment}-\eqref{eq:B-unexpanded}.
\end{proof}
%=====================================================
\section{Positivity of the dissipation and proof of the refined condition}\label{sec:positivity}
%=====================================================
The sign of $\widetilde B_k$ is determined by a scalar $2\times2$ matrix.
This reduction does not require the Hilbert space to be finite-dimensional.

The relevant determinant condition gives a direct test for positivity.

\begin{proposition}
\label{prop:positivity}
With the notation of \eqref{eq:dimensionless},
\begin{equation}
\label{eq:B-expanded}
2L\widetilde B_k
=(\tau-1)\norm{X}^2+\tau C_q\norm{Y}^2
+2\bigl(\tau\delta-(\tau-1)\bigr)\inner{X}{Y}.
\end{equation}
For $\tau>1$, the coefficient matrix of this quadratic form is positive
semidefinite if and only if
\begin{equation}
\label{eq:finite-tau-condition}
\bigl(\tau(\delta-1)+1\bigr)^2
\leq\tau(\tau-1)C_q.
\end{equation}
If
\begin{equation}
\label{eq:limit-condition}
(\delta-1)^2<C_q,
\end{equation}
then \eqref{eq:finite-tau-condition} holds for every sufficiently large
$\tau$.
\end{proposition}

\begin{proof}
Expanding $\norm{X-Y}^2$ in \eqref{eq:B-unexpanded} and using
$Ls=q$ and $\delta=\beta L\sqrt{s}$ gives \eqref{eq:B-expanded}. The
coefficient matrix is
\begin{equation}
\label{eq:M-tau}
M_\tau=
\begin{pmatrix}
\tau-1 & \tau\delta-(\tau-1)\\
\tau\delta-(\tau-1) & \tau C_q
\end{pmatrix}.
\end{equation}
For $\tau>1$, both diagonal entries are positive. Hence $M_\tau$ is
positive semidefinite if and only if its determinant is nonnegative, which
is \eqref{eq:finite-tau-condition}. After division by $\tau^2$, the two
sides of \eqref{eq:finite-tau-condition} converge to $(\delta-1)^2$ and
$C_q$, respectively. Condition \eqref{eq:limit-condition} therefore implies
\eqref{eq:finite-tau-condition} for all sufficiently large $\tau$.
\end{proof}

The strict inequality makes the dissipation coercive for all sufficiently large indices.

\begin{corollary}
\label{cor:coercivity}
Suppose that \eqref{eq:limit-condition} holds. There exist $c>0$ and
$k_0\geq1$ such that, for every $k\geq k_0$,
\begin{equation}
\label{eq:coercivity}
\widetilde B_k
\geq c\tau_{k+1}
\left(\norm{\nabla f(x_k)}^2+\norm{\nabla f(y_k)}^2\right).
\end{equation}
\end{corollary}

\begin{proof}
Dividing \eqref{eq:M-tau} by $\tau$ and letting $\tau\to+\infty$ gives
$$
\frac{M_\tau}{\tau}\longrightarrow
M_\infty:=
\begin{pmatrix}
1 & \delta-1\\
\delta-1 & C_q
\end{pmatrix}.
$$The first leading principal minor of $M_\infty$ is $1$, and
$$
\det M_\infty=C_q-(\delta-1)^2>0.
$$
Thus $M_\infty$ is positive definite. Let
$\lambda_\infty>0$ denote its smallest eigenvalue. For all sufficiently
large $\tau$,
$$
\lambda_{\min}\left(\frac{M_\tau}{\tau}\right)
\geq\frac{\lambda_\infty}{2}.
$$
Let $\lambda_1(\tau)$ and $\lambda_2(\tau)$ be the eigenvalues of
$M_\tau$, and choose an orthogonal matrix
$P_\tau=(p_{ij})_{1\leq i,j\leq2}$ such that
$$
P_\tau^{\mathsf T}M_\tau P_\tau
=
\begin{pmatrix}
\lambda_1(\tau)&0\\
0&\lambda_2(\tau)
\end{pmatrix}.
$$
For $X,Y\in\Hc$, set
$$
U=p_{11}X+p_{21}Y,
\qquad
V=p_{12}X+p_{22}Y.
$$
The orthogonality of $P_\tau$ gives
$$
\norm{U}^2+\norm{V}^2
=
\norm{X}^2+\norm{Y}^2.
$$
Moreover, the quadratic form in \eqref{eq:B-expanded} becomes
$$
2L\widetilde B_k
=
\lambda_1(\tau)\norm{U}^2
+\lambda_2(\tau)\norm{V}^2.
$$
Consequently, for all sufficiently large $\tau$,
$$
2L\widetilde B_k
\geq
\lambda_{\min}(M_\tau)
\bigl(\norm{U}^2+\norm{V}^2\bigr)\geq
\frac{\lambda_\infty}{2}\tau
\bigl(\norm{X}^2+\norm{Y}^2\bigr).
$$
Taking $\tau=\tau_{k+1}$ proves \eqref{eq:coercivity} with $
c=\frac{\lambda_\infty}{4L},$ which completes the proof.
\end{proof}

We can now prove Theorem~\ref{thm:main} by combining the energy estimate and asymptotic coercivity.

\begin{proof}[Proof of Theorem~\ref{thm:main}]
For $q\in(0,1]$, we have $C_q\geq1$. Since $\delta>0$, condition
\eqref{eq:main-condition} is equivalent to
$(\delta-1)^2<C_q$. Proposition~\ref{prop:energy-increment} and
Corollary~\ref{cor:coercivity} yield constants $c>0$ and $k_0$ such that
\begin{equation}
\label{eq:energy-decay}
\mathcal E_{k+1}-\mathcal E_k
\leq-c\tau_{k+1}^2
\left(\norm{\nabla f(x_k)}^2+\norm{\nabla f(y_k)}^2\right)
\leq0
\end{equation}
for every $k\geq k_0$. We may increase $k_0$ so that $\tau_k>0$ on this
range. Since $\mathcal E_k\geq0$, eventual monotonicity gives
$$
0\leq f(x_k)-f^\star
\leq\frac{\mathcal E_k}{\tau_k^2}
\leq\frac{\mathcal E_{k_0}}{\tau_k^2}.
$$
Because $\tau_k=(k-1)/(\alpha-1)$, this proves
\eqref{eq:objective-rate}.

Summing \eqref{eq:energy-decay} from $k_0$ to $N$ gives
$$
c\sum_{k=k_0}^{N}\tau_{k+1}^2
\left(\norm{\nabla f(x_k)}^2+\norm{\nabla f(y_k)}^2\right)
\leq\mathcal E_{k_0}-\mathcal E_{N+1}
\leq\mathcal E_{k_0}.
$$
Letting $N\to+\infty$ and using
$\tau_{k+1}=k/(\alpha-1)$ proves \eqref{eq:gradient-summability}. The
finitely many indices before $k_0$ do not affect convergence. The terms of
each convergent nonnegative series tend to zero. Hence
$$
k^2\norm{\nabla f(x_k)}^2\longrightarrow0,
\qquad
k^2\norm{\nabla f(y_k)}^2\longrightarrow0,
$$
which proves \eqref{eq:gradient-little-o}.

Finally, let
$$
C_x:=\sum_{i=1}^{+\infty}i^2\norm{\nabla f(x_i)}^2.
$$
Then
$$
\min_{1\leq i\leq k}\norm{\nabla f(x_i)}^2
\sum_{i=1}^ki^2\leq C_x.
$$
Since $\sum_{i=1}^ki^2=k(k+1)(2k+1)/6$, the first estimate in
\eqref{eq:best-iterate} follows. The argument for $(y_i)$ is identical.
\end{proof}

The same calculation also gives the additional estimate required in \cite{ACFR2023}. The following remark clarifies this connection.

\begin{remark}\normalfont
\label{rem:acfr2023-extension}
Assume in addition that $\alpha>3$. In the passage from
\eqref{eq:weighted-function} to \eqref{eq:weighted-function-final}, retain
the difference
$$
\omega_k
:=\tau_k^2-\tau_{k+1}(\tau_{k+1}-1)
=\frac{k(\alpha-3)+1}{(\alpha-1)^2}.
$$
The calculation in Proposition~\ref{prop:energy-increment} then gives, for
all sufficiently large $k$,
$$
\mathcal E_{k+1}-\mathcal E_k
+\omega_k\bigl(f(x_k)-f^\star\bigr)
\leq-\tau_{k+1}\widetilde B_k.
$$
Under \eqref{eq:main-condition}, Corollary~\ref{cor:coercivity} gives
$\widetilde B_k\geq0$ for all sufficiently large $k$. After adding the
finite initial part, summation yields
$$
\sum_{k=1}^{+\infty}
\omega_k\bigl(f(x_k)-f^\star\bigr)<+\infty.
$$
Since $
\omega_k\geq\frac{\alpha-3}{(\alpha-1)^2}\,k,$ we obtain
$$
\sum_{k=1}^{+\infty}k\bigl(f(x_k)-f^\star\bigr)<+\infty.
$$
Together with Theorem~\ref{thm:main}, this recovers under
\eqref{eq:main-condition} the estimates stated in
\cite[Theorem~4.1]{ACFR2023}. After omitting finitely many initial indices,
the proof of \cite[Theorem~4.2]{ACFR2023} gives the same velocity estimates.
These estimates, the boundedness of $(\mathcal E_k)$, and the weighted
gradient summability are the inputs of the Opial argument in
\cite[Theorem~4.3]{ACFR2023}; no further restriction on $\beta$ is used.
Hence $(x_k)$ converges weakly to a point of $X^\star$. The proof of
\cite[Theorem~4.4]{ACFR2023} also applies. All these convergence arguments
remain those of \cite{ACFR2023}.
\end{remark}

A simpler sufficient condition follows directly from the refined parameter bound.

\begin{corollary}\label{cor:simple-condition}
Let $(x_k)$ be generated by \eqref{eq:igahd}, with $\alpha\geq3$,
$0<s\leq1/L$, and $\beta>0$. If
$$
0<\beta L\sqrt{s}<2,
$$
then all the conclusions of Theorem~\ref{thm:main} hold.
\end{corollary}

\begin{proof}
For $q\in(0,1]$, $q(1-q)\geq0$, and therefore
$1+\sqrt{1+q(1-q)}\geq2$.
\end{proof}

The index from which positivity holds can also be estimated explicitly.
\begin{remark}
Assume that $\delta>0$.
Let $
D:=C_q-(\delta-1)^2>0$ and $
P(\tau):=
D\tau^2-\bigl(C_q+2(\delta-1)\bigr)\tau-1.$
Since, $
P(1)=-\delta^2<0,$
and $D>0$, the positive root $\tau_{\mathrm{pos}}$ of $P$ satisfies
$\tau_{\mathrm{pos}}>1$. Moreover, $P(\tau)\geq0$ whenever
$$
\tau\geq\tau_{\mathrm{pos}}
:=
\frac{C_q+2(\delta-1)
+\sqrt{\bigl(C_q+2(\delta-1)\bigr)^2+4D}}{2D}.
$$
Therefore, for every $k$ such that
$$
\tau_{k+1}=\frac{k}{\alpha-1}\geq\tau_{\mathrm{pos}},
$$
the matrix $M_{\tau_{k+1}}$ is positive semidefinite, and hence
$\widetilde B_k\geq0$. In particular, the Lyapunov energy is nonincreasing
from a finite index onward. The asymptotic estimates require only eventual
positivity and coercivity: the finitely many preceding indices are absorbed
into the constants. Consequently, no explicit value of this index is needed
for the convergence results.
\end{remark}
%=====================================================
\section{Comparison of the parameter regions}\label{sec:comparison}
%=====================================================
\subsection{Lyapunov parameter regions}

In terms of $q=sL$ and $\delta=\beta L\sqrt{s}$, the condition from
\cite[Theorem~6]{AFCR2022} is
\begin{equation}
\label{eq:old-condition}
0\leq\beta<2\sqrt{s}
\quad\Longleftrightarrow\quad
0\leq\delta<2q.
\end{equation}
For $\beta>0$, the original condition becomes $0<\delta<2q$. The simple
condition in Corollary~\ref{cor:simple-condition} is $0<\delta<2$, while the
refined condition is $0<\delta<1+\sqrt{1+q(1-q)}$. Since $0<q\leq1$,
$$
2q\leq2\leq1+\sqrt{1+q(1-q)}.
$$
Thus, in the positive-damping regime, the refined condition contains the
original one, and the inclusion is strict when $s<1/L$.
\begin{table}[htbp]
\centering
\caption{Comparison of the parameter conditions, with
$q=sL$ and $\delta=\beta L\sqrt{s}$.}
\label{tab:parameter-comparison}
\begin{tabularx}{\textwidth}{
@{}
p{0.20\textwidth}
>{\centering\arraybackslash}p{0.30\textwidth}
>{\centering\arraybackslash}X
@{}
}
\toprule
Parameter condition & Dimensionless form & Bound on $\beta$ \\
\midrule
Original
& $0\leq\delta<2q$
& $0\leq\beta<2\sqrt{s}$ \\

Simpler sufficient
& $0<\delta<2$
& $0<\beta<2/(L\sqrt{s})$ \\

Refined
& $0<\delta<1+\sqrt{1+q(1-q)}$
& \mbox{$0<\beta<
\dfrac{1+\sqrt{1+sL(1-sL)}}{L\sqrt{s}}$} \\
\bottomrule
\end{tabularx}
\end{table}
%%%%%%%%%%%%
For fixed $L$ and $\beta>0$, the difference becomes pronounced as
$s\to0$. The original upper bound satisfies $2\sqrt{s}\to0$, whereas
$$
\frac{1+\sqrt{1+sL(1-sL)}}{L\sqrt{s}}
\sim\frac{2}{L\sqrt{s}}\longrightarrow+\infty.
$$
Every fixed $\beta>0$ therefore satisfies the refined condition once $s$ is
sufficiently small. At the endpoint $s=1/L$, i.e. $q=1$, both margins in
\eqref{eq:margins} vanish, and the refined condition reduces to the original
one.

{
\subsection{Spectral analysis for convex quadratic objectives}
\label{subsec:quadratic}

{Under the original condition $0<\beta<2\sqrt{s}$ and for
$\alpha>3$, \cite[Theorem~4.3]{ACFR2023} proves weak convergence of the
iterates for smooth convex functions on a Hilbert space. The argument below is
restricted to finite-dimensional convex quadratic functions. It extends the
known convergence result to a parameter range larger than the refined
Lyapunov region, includes $\alpha=3$, and gives geometric decay of the
objective residual and the gradients. It does not prove that the energy
$\mathcal E_k$ is nonincreasing.}

Consider the function
$$
f_Q(x)=\frac12\langle Q(x-x^\star),x-x^\star\rangle+f^\star,
\qquad
Q=Q^\top\succeq0,
\qquad
\norm{Q}\leq L,
$$
on $\R^n$. Here, $\norm{Q}$ denotes the operator norm induced by the Euclidean norm. Let
$$
Q=U\operatorname{diag}(\lambda_1,\ldots,\lambda_n)U^\top,
\qquad
0\leq\lambda_j\leq L,
$$
where $U$ is orthogonal, and introduce
$$
\xi_k:=U^\top(x_k-x^\star),
\qquad
\eta_k:=U^\top(y_k-x^\star).
$$
For each eigencoordinate $j$, set $q_j:=s\lambda_j, \mbox{ and }d_j:=\beta\lambda_j\sqrt{s}.$
Substitution in \eqref{eq:igahd} gives, for each coordinate $j$,
\begin{equation}
\label{eq:quadratic-mode}
\xi_{k+1,j}=(1-q_j)
\left[
\left(2-d_j-\frac{\alpha}{k}\right)\xi_{k,j}
+\left(-(1-d_j)+\frac{\alpha-d_j}{k}\right)\xi_{k-1,j}
\right].
\end{equation}
The coefficients in this recurrence depend on $k$.
For an index $j$ with $q_j>0$, define
$$
Z_{k,j}:=
\begin{pmatrix}
\xi_{k,j}\\
\xi_{k-1,j}
\end{pmatrix}.
$$
Equation~\eqref{eq:quadratic-mode} is equivalent to the exact system: $
Z_{k+1,j}=A_{k,j}Z_{k,j},$
where
$$
A_{k,j}:=
\begin{pmatrix}
(1-q_j)\left(2-d_j-\dfrac{\alpha}{k}\right)
&
(1-q_j)\left(-(1-d_j)+\dfrac{\alpha-d_j}{k}\right)\\
1&0
\end{pmatrix},
$$
and
$$
A_{k,j}\longrightarrow
A_{\infty,j}:=
\begin{pmatrix}
(1-q_j)(2-d_j)&-(1-q_j)(1-d_j)\\
1&0
\end{pmatrix}.
$$
The matrix $A_{\infty,j}$ is the limit of the coefficient matrices. Its
characteristic polynomial is
\begin{equation}
\label{eq:quadratic-polynomial}
p_j(r)=r^2-(1-q_j)(2-d_j)r+(1-q_j)(1-d_j).
\end{equation}
A matrix is Schur stable if all its eigenvalues have modulus less than one.
For a real quadratic polynomial $p(r)=r^2-Ar+B$, the degree-two case of the
Jury criterion states that both roots have modulus less than one if and only
if
$$
1-A+B>0,
\qquad
1+A+B>0,
\qquad
1-B>0;
$$
see \cite[p.~146, case $n=2$]{Jury}. For
\eqref{eq:quadratic-polynomial}, these quantities are
$$
1-A+B=q_j,
\qquad
1+A+B=1+(1-q_j)(3-2d_j),
\qquad
1-B=q_j+d_j-q_jd_j.
$$
If $0<q_j<1$ and $d_j\geq0$, the first and third quantities are positive.
Therefore, only the second condition imposes a restriction. The matrix
$A_{\infty,j}$ is Schur stable if and only if
\begin{equation}
\label{eq:exact-modal}
d_j<\frac{4-3q_j}{2(1-q_j)}.
\end{equation}
Since $d_j=q_j\beta/\sqrt{s}$, define
$$
\phi(q):=\frac{4-3q}{2q(1-q)},
\qquad 0<q<1.
$$
For a fixed matrix $Q$, all limiting matrices with $0<q_j<1$ are Schur stable
if and only if
\begin{equation}
\label{eq:fixed-quadratic}
\frac{\beta}{\sqrt{s}}<
\min_{\substack{\lambda\in\sigma(Q)\\0<s\lambda<1}}
\phi(s\lambda),
\end{equation}
where the minimum is interpreted as $+\infty$ when the indexing set is empty.

This condition gives geometric decay of the positive spectral modes. The zero
modes are treated separately because they determine the limit of the
iterates.

\begin{proposition}
\label{prop:fixed-quadratic}
Let $\alpha\geq3$, $0<s\leq1/L$, and $\beta\geq0$, and let $Q$ be a fixed
symmetric positive semidefinite matrix with $\norm{Q}\leq L$. The limiting
matrices associated with all positive spectral modes are Schur stable if and
only if \eqref{eq:fixed-quadratic} holds.

If this condition holds, then, for every initial pair $x_0,x_1\in\R^n$,
there exist $C>0$, $\rho\in(0,1)$, and $k_0\geq1$ such that, for every
$k\geq k_0$,
\begin{equation}
\label{eq:quadratic-geometric}
0\leq f_Q(x_k)-f^\star\leq C\rho^{2k}
\end{equation}
and
\begin{equation}
\label{eq:quadratic-gradient-geometric}
\norm{\nabla f_Q(x_k)}+\norm{\nabla f_Q(y_k)}
\leq C\rho^k.
\end{equation}
The constants may depend on $Q$ and on the initial pair.
Consequently, the estimates
\eqref{eq:objective-rate}-\eqref{eq:best-iterate} hold for $f_Q$.
Moreover, $(x_k)$ converges to a minimizer of $f_Q$.
\end{proposition}

\begin{proof}
For $0<q_j<1$, the equivalence follows from
\eqref{eq:exact-modal}. If $q_j=1$, the limiting matrix is nilpotent and is
therefore Schur stable for every $d_j\geq0$.

Suppose that \eqref{eq:fixed-quadratic} holds, and
fix an index $j$ with $0<q_j<1$. Since $A_{\infty,j}$ is Schur
stable, a standard result in
finite-dimensional matrix analysis gives a norm $\norm{\cdot}_{\ast,j}$ on
$\R^2$ such that the induced matrix norm satisfies
$$
\norm{A_{\infty,j}}_{\ast,j}<1.
$$
All matrix norms on $\R^{2\times2}$ are equivalent. Hence
$A_{k,j}\to A_{\infty,j}$ also holds in this induced norm. We may therefore
choose $\rho_j\in(0,1)$ and $k_j\geq1$ such that $
\norm{A_{k,j}}_{\ast,j}\leq\rho_j$
for every $k\geq k_j$. The exact state equation then gives
$$
\norm{Z_{k+1,j}}_{\ast,j}
\leq\rho_j\norm{Z_{k,j}}_{\ast,j},
\quad \mbox{ and therefore }
\quad
\norm{Z_{k,j}}_{\ast,j}
\leq \rho_j^{k-k_j}\norm{Z_{k_j,j}}_{\ast,j}
$$
for every $k\geq k_j$. Equivalence of norms on $\R^2$ shows that
$\xi_{k,j}$ decays geometrically. The gradient step gives the exact identity $
\xi_{k+1,j}=(1-q_j)\eta_{k,j}.$
Since $q_j<1$, $\eta_{k,j}$ also decays geometrically.

If $q_j=1$, \eqref{eq:quadratic-mode} gives $\xi_{k+1,j}=0$ for every $k\geq1$. Hence $\xi_{k,j}=0$ for $k\geq2$. The definition of $y_k$ then
gives $\eta_{k,j}=0$ for $k\geq3$. If the set of indices $j$ such that $0<q_j<1$ is nonempty, we take the largest of the corresponding contraction factors. If this set is empty,
every positive spectral component corresponds to $q_j=1$ and therefore vanishes after finitely many iterations; the assertion is also trivial
when $Q$ has no positive eigenvalue. Increasing the constant then gives common geometric bounds.

For a zero mode, \eqref{eq:quadratic-mode} reduces to
$$
\Delta\xi_{k+1,j}
=\left(1-\frac{\alpha}{k}\right)\Delta\xi_{k,j},
\qquad
\Delta\xi_{k,j}:=\xi_{k,j}-\xi_{k-1,j}.
$$
Choose an integer $k_1>\alpha$. For $k\geq k_1$, all factors below are
positive. Iteration and the inequality $\log(1-u)\leq-u$ for $u\in(0,1)$
give
$$
\lvert\Delta\xi_{k,j}\rvert
=\lvert\Delta\xi_{k_1,j}\rvert
\prod_{\ell=k_1}^{k-1}\left(1-\frac{\alpha}{\ell}\right)
=O(k^{-\alpha}).
$$
Since $\alpha>1$, the sequence of increments is summable. Thus every nullspace component converges. The positive components of $x_k-x^\star$
converge to zero. It follows that $(x_k)$ converges to a point $\bar x$ such that $Q(\bar x-x^\star)=0$. This point is a minimizer of $f_Q$.

Finally,
$$
f_Q(x_k)-f^\star
=\frac12\sum_{j=1}^n\lambda_j\lvert\xi_{k,j}\rvert^2, \qquad
\norm{\nabla f_Q(x_k)}^2
=\sum_{j=1}^n\lambda_j^2\lvert\xi_{k,j}\rvert^2,
\qquad
\norm{\nabla f_Q(y_k)}^2
=\sum_{j=1}^n\lambda_j^2\lvert\eta_{k,j}\rvert^2.
$$
The zero modes do not contribute to these expressions. The geometric bounds
on the positive modes give \eqref{eq:quadratic-geometric} and
\eqref{eq:quadratic-gradient-geometric}. Geometric decay implies
\eqref{eq:objective-rate}-\eqref{eq:best-iterate}.
\end{proof}

The condition for a fixed matrix depends on its positive eigenvalues. Taking
the infimum over all possible eigenvalues gives a condition that is uniform
over the class $Q\succeq0$, $\norm{Q}\leq L$.

\begin{proposition}
\label{prop:uniform-quadratic}
Fix $\alpha\geq3$, $0<s\leq1/L$, and $\beta\geq0$, and set
$$
q_{\max}:=sL\in(0,1].
$$
The following statements are equivalent:
\begin{enumerate}[leftmargin=*,label=\textup{(\roman*)}]
\item For every symmetric positive semidefinite matrix $Q$ satisfying
$\norm{Q}\leq L$, the limiting matrix associated with each positive spectral
mode is Schur stable.
\item The parameter satisfies
\begin{equation}
\label{eq:uniform-quadratic}
\frac{\beta}{\sqrt{s}}<T(q_{\max}):=
\begin{cases}
\dfrac{4-3q_{\max}}{2q_{\max}(1-q_{\max})},
&0<q_{\max}\leq\dfrac23,\\[2mm]
\dfrac92,&\dfrac23\leq q_{\max}\leq1.
\end{cases}
\end{equation}
\end{enumerate}
Whenever these conditions hold, the conclusions of
Proposition~\ref{prop:fixed-quadratic} apply to every such fixed matrix $Q$.
\end{proposition}

\begin{proof}
For $q=s\lambda\in(0,1)$, condition \eqref{eq:exact-modal} is equivalent to $
\frac{\beta}{\sqrt{s}}<\phi(q).
$
The endpoint $q=1$ is Schur stable without any condition on $\beta$. Every $q\in(0,q_{\max}]$ occurs in the class as $q=s\lambda$ for a
positive eigenvalue $\lambda$; for example, take $Q=(q/s)I$. Hence statement
\textup{(i)} is equivalent
to
$$
\frac{\beta}{\sqrt{s}}
<\inf_{\substack{0<q\leq q_{\max}\\q<1}}\phi(q).
$$
Differentiation gives
$$
\phi'(q)=
\frac{(3q-2)(2-q)}{2q^2(1-q)^2}.
$$
Thus $\phi$ decreases on $(0,2/3]$ and increases on $[2/3,1)$. Its minimum
is $\phi(2/3)=9/2$. Evaluation of the infimum gives
\eqref{eq:uniform-quadratic}. Proposition~\ref{prop:fixed-quadratic} gives the
last assertion.
\end{proof}

Proposition~\ref{prop:uniform-quadratic} gives a condition on $\beta$ that is uniform in $Q$. It does not give geometric constants that are uniform in
$Q$. These constants may also depend on the initial pair. In particular, a contraction factor strictly below one cannot be uniform when positive
eigenvalues may approach zero.

We now compare the four conditions in the same variables. Set $q=q_{\max}=sL$ and define
$$
\Psi(q):=\frac{1+\sqrt{1+q(1-q)}}{q}.
$$
In terms of $\beta/\sqrt{s}$, the original condition, the simpler sufficient
condition, and the refined Lyapunov condition are, respectively,
$$
\frac{\beta}{\sqrt{s}}<2,
\qquad
\frac{\beta}{\sqrt{s}}<\frac{2}{q},
\qquad
\frac{\beta}{\sqrt{s}}<\Psi(q).
$$
The spectral condition is $\beta/\sqrt{s}<T(q)$. For $0<q\leq2/3$, the
inequality $T(q)>\Psi(q)$ reduces to
$$
\frac{2-q}{2(1-q)}>\sqrt{1+q(1-q)}.
$$
Both sides are positive, and the difference of their squares is
$$
\frac{(2-q)^2}{4(1-q)^2}-\bigl(1+q(1-q)\bigr)
=\frac{q^2(3-2q)^2}{4(1-q)^2}>0.
$$
For $2/3\leq q\leq1$, differentiation gives
$$
\Psi'(q)=\frac{1}{q^2}
\left(
\frac{q(1-2q)}{2\sqrt{1+q(1-q)}}
-1-\sqrt{1+q(1-q)}
\right)<0.
$$
It follows that
$$
\Psi(q)\leq\Psi(2/3)
=\frac{3+\sqrt{11}}{2}<\frac92=T(q).
$$
Also, $2\leq2/q\leq\Psi(q)$. We have proved that
$$
2\leq\frac2q\leq\Psi(q)<T(q),
\qquad 0<q\leq1.
$$

{Thus, on convex quadratic functions, the spectral condition
extends the parameter range in which Proposition~\ref{prop:fixed-quadratic}
gives convergence of the iterates and geometric estimates.}

The bound $\beta/\sqrt{s}<T(q)$ is necessary and sufficient for Schur
stability of all the limiting matrices in the quadratic class. This is the
only necessity statement proved here. We do not claim that this bound is
necessary for convergence of the IGAHD recurrence with variable coefficients.
This restriction also applies to the boundary case, where the limiting
characteristic polynomial has a root on the unit circle.
}
%=====================================================
\section{Boundary cases and scope}\label{sec:scope}
%=====================================================
\paragraph{The case $\beta=0$.}
Theorem~\ref{thm:main} is stated for $\beta>0$ because it concerns the
Hessian-damped regime. When $\beta=0$, the accelerated objective estimate
follows from the standard argument. If $s<1/L$, then
$(\delta-1)^2=1<C_q$, so the refined calculation also gives the weighted
gradient summability. At $s=1/L$, the limit matrix is positive semidefinite
but not positive definite, and the coercivity argument used here does not
apply.

{
\paragraph{The boundary of the Lyapunov condition.}
When $(\delta-1)^2=C_q$, the matrix $M_\infty$ is singular. The present
Lyapunov argument then loses uniform coercivity. This does not imply
divergence on the boundary, and it does not exclude a proof based on another
estimate. In the regime $\delta>0$ of Theorem~\ref{thm:main}, equality
corresponds to the upper branch $\delta=1+\sqrt{C_q}$. The determinant
condition used in the proof then fails for every $\tau>1$. This case is
different from the degenerate case $\delta=0$, $q=1$.
}

{
\paragraph{The spectral threshold.}
For the quadratic class, Proposition~\ref{prop:uniform-quadratic} gives a
necessary and sufficient condition for Schur stability of all limiting modal
matrices. This statement concerns the matrices obtained by passing to the
limit in the coefficients of the recurrence. It does not give a necessary
condition for convergence of the original recurrence, whose coefficients
depend on $k$. On the boundary of the uniform condition, the quadratic class
contains a mode whose limiting matrix has an eigenvalue of modulus one. The
contraction argument used in Proposition~\ref{prop:fixed-quadratic} does not
apply to this mode. A separate argument would be needed to decide convergence
or divergence in this case.
}

\paragraph{Eventual monotonicity.}
The determinant condition holds for all sufficiently large
$\tau_{k+1}$. Consequently, the energy is eventually nonincreasing. This is
enough for the asymptotic estimates because the preceding finite set of
indices is absorbed into the constants. The proof of
\cite[Theorem~6]{AFCR2022} likewise establishes its
determinant condition for sufficiently large indices and uses
$t_{k+1}-1\geq0$. We therefore state only the eventual monotonicity that
follows directly from these Lyapunov calculations; no claim of monotonicity
from the initial index is made here.

{
\paragraph{Lack of uniformity as $s\to0$.}
For each fixed $\beta>0$, the refined condition holds when $s>0$ is
sufficiently small. The Lyapunov estimates are not uniform in this limit.
Indeed, $q=sL\to0$ and $\delta=\beta L\sqrt{s}\to0$, so
$$
M_\infty
=\begin{pmatrix}1&\delta-1\\ \delta-1&1+q(1-q)\end{pmatrix}
\longrightarrow
\begin{pmatrix}1&-1\\ -1&1\end{pmatrix},
$$
which is singular. Hence the coercivity constant in
Corollary~\ref{cor:coercivity} degenerates. The theorem gives a valid sufficient condition for every fixed sufficiently small $s$, but it is not a
uniform-in-$s$ discretization result.
}

{
\paragraph{Scope and attribution.}
The algorithm, the scaling factors $\tau_k$, the auxiliary vector $v_k$, the
Lyapunov function, and the cancellation identity are due to
\cite{AFCR2022}. The Lyapunov refinement consists in retaining the exact
coefficients in one standard smooth convex estimate. It gives a less
restrictive sufficient condition for the conclusions of the original
Lyapunov analysis. No claim of necessity or optimality is made on the full
class of smooth convex functions.

}
%=====================================================
\section{Conclusion}\label{sec:conclusion}
%=====================================================
The algorithm, auxiliary sequence, Lyapunov function, and cancellation identity considered here are taken from \cite{AFCR2022}. The part of this note devoted to Lyapunov analysis refines an estimate found in their proof. Retaining two exact coefficients instead of replacing them with $s/2$ gives the less restrictive sufficient condition
$$
0<\beta L\sqrt{s}<1+\sqrt{1+sL(1-sL)}
$$
within the same Lyapunov construction. The simpler condition $0<\beta L\sqrt{s}<2$ is sufficient. Thus, for a fixed $\beta>0$, the refined condition holds for any sufficiently small step size. 

For $\alpha>3$, the convergence arguments from \cite{ACFR2023} also apply under the refined condition. In particular, $(x_k)$ converges weakly to a point in $X^\star$.

In the case of convex quadratic objective functions in finite dimensions, spectral analysis extends the convergence result of \cite[Theorem 4.3]{ACFR2023} to the broader parameter domain \eqref{eq:uniform-quadratic}, including the case $\alpha=3$. It also establishes geometric decay of the objective function residual and the gradients. The statement regarding necessity in this analysis concerns only the Schur stability of the limiting modal matrices.

This analysis raises two natural questions. The first is to determine the behavior of the non-autonomous quadratic recurrence at the boundary of the Schur modal domain. The second is whether a different Lyapunov function would allow recovering part (or the whole) of this expanded spectral domain for non-quadratic objective functions.

%======================

%=====================================================
\section*{Acknowledgements}
%=====================================================
The author gratefully acknowledges support from the Math AmSud project
N\textdegree~51756TF (VIPS), ECOS Project C24E06, and the FMJH Gaspard Monge
Program for Optimization and Data Science.
\end{document}